\documentclass[12pt,leqno]{amsart}%
\allowdisplaybreaks
\usepackage{amsmath}
\usepackage{mathtools}
\usepackage{esint}

\usepackage[margin=1.5in]{geometry}
\usepackage{graphicx}
\usepackage{amsfonts}
\usepackage{amssymb}
\usepackage[normalem]{ulem}
\usepackage[utf8]{inputenc}
\usepackage{comment}
\usepackage[initials,msc-links]{amsrefs}
\usepackage{hyperref}%
\usepackage{amsthm}
\usepackage{url}
\usepackage[dvipsnames]{xcolor}
\usepackage[english=american]{csquotes}
\usepackage{enumerate}
\usepackage{textcomp}
\usepackage{mathrsfs}
\usepackage{mathtools}
\usepackage{color, colortbl}
\definecolor{Gray}{gray}{0.9}
\usepackage{bbm}
\usepackage{stmaryrd}
\providecommand{\U}[1]{\protect\rule{.1in}{.1in}}
\newtheorem{theorem}{Theorem}[section]
\theoremstyle{plain}

\newtheorem{claim}{Claim}

\newtheorem{corollary}{Corollary}[section]

\newtheorem{lemma}{Lemma}[section]

\newtheorem{proposition}{Proposition}[section]

\numberwithin{equation}{section}
\theoremstyle{definition}
\newtheorem{definition}{Definition}[section]
\theoremstyle{remark}
\newtheorem{remark}{Remark}[section]

\newcommand{\note}[1]{\strut{\color{red}[#1]}}

\newcommand{\R}{\mathbb{R}}

\newcommand{\Hcal}{\mathcal{H}}

\newcommand{\BMO}{\mathrm{BMO}}

\begin{document}
\title[Variational integrals on Hessian spaces]{Variational integrals on Hessian spaces: partial regularity for critical points}
\author{Arunima Bhattacharya}
\address{Department of Mathematics, Phillips Hall\\
 the University of North Carolina at Chapel Hill, NC }
\email{arunimab@unc.edu}

\author{Anna Skorobogatova
}
\address{Department of Mathematics \\
ETH Z\"{u}rich \\ R\"{a}mistrasse 101 \\ 8092 Z\"{u}rich}
\email{anna.skorobogatova@math.ethz.ch}

 \begin{abstract}
     We develop regularity theory for critical points of variational integrals defined on Hessian spaces of functions on open, bounded subdomains of $\mathbb{R}^n$, under compactly supported variations. 
     %The critical point solves a fourth order nonlinear equation in double divergence form. 
     We show that for smooth convex functionals, a $W^{2,\infty}$ critical point with bounded Hessian is smooth provided that its Hessian has a small bounded mean oscillation (BMO). We deduce that the interior singular set of a critical point has Hausdorff dimension at most $n-p_0$, for some $p_0 \in (2,3)$. 
    We state some applications of our results to variational problems in Lagrangian geometry. Finally, we use the Hamiltonian stationary equation to demonstrate the importance of our assumption on the a priori regularity of the critical point.
 \end{abstract}

 \maketitle

	\section{ Introduction}

We consider variational integrals of the form
 \begin{equation}
\int_{\Omega}F(D^{2}u)\,dx\label{Ffunc}%
\end{equation} where $\Omega$ is an open, bounded subset of $\mathbb{R}^n.$
Any critical point of the above functional under compactly supported variations satisfies an Euler-Lagrange equation, which possesses a nonlinear fourth order structure, given by 
\begin{equation}
\int_{\Omega}  F^{ij}(D^{2}u)\eta_{ij}\,dx=0, \label{main_0}
\end{equation}
where $\eta\in C_{0}^{\infty}(\Omega)$ is an arbitrary test function, {and we are using the notation $\eta_{ij} = \partial_{ij}\eta$ and $F^{ij}(D^2 u) = \frac{\partial F}{\partial u_{ij}}(D^2 u)$}. We assume the functional in (\ref{Ffunc}) to be uniformly convex and smooth, which results in the above coefficient $F^{ij}$ to possess the following properties:
\begin{enumerate}
    \item It is smooth in the matrix entries $D^{2}u$
over a given convex region $U\subset S^{n\times n}$, where $S^{n\times n}$ denotes the space of all $n\times n$ symmetric matrices.
    \item It satisfies the Legendre ellipticity condition: there exists a constant $\Lambda>0$ such that
\begin{equation*}
\label{elliptic:Intro}\frac{\partial F^{ij}}{\partial u_{kl}}(\xi)\sigma
_{ij}\sigma_{kl}\geq\Lambda\vert \sigma\vert ^{2},\text{ $\forall$
}\sigma\text{ $\in S^{n\times n}$, $\forall \ \xi\in U$}.
\end{equation*}
\end{enumerate}
{One may equivalently replace the uniform convexity with uniform concavity, which requires $-F$ to be uniformly convex.} 

Hessian-dependent functionals appear in various areas of mathematics and related
disciplines. Examples include functionals arising in elasticity theory
and the mechanics of solids \cite{KV}, as well as the Aviles-Giga functional \cite{AG} that models phenomena from blistering to liquid crystals. Other equations that enjoy the fourth order structure are bi-harmonic functions, in addition to suitable perturbations of this, such as the conformally invariant Paneitz operator introduced in \cite{P} and the wider class of operators studied in \cite{CGY}. Further, well-known examples are extremal K\"ahler metrics, the Willmore surface equation, which is closely linked to elastic mechanics, and its generalization coming from the Canham-Helfrich energy \cite{MonSch}.\\

For second order equations in divergence form (and their nonlinear counterparts), the existence and regularity theory of solutions has been studied extensively and plays a significant role in geometric analysis, among many other fields. In comparison, the theory of double divergence form equations is largely unexplored for fourth order but remains an important developing area of geometric analysis. \\

	Before presenting the main result of this paper, we first introduce the following notation and definition. \\

 \noindent
 \textbf{Notations.} 
 Throughout this paper, $U$ denotes a convex neighborhood in $S^{n\times n}$.
 $B_r(x)$ denotes an open ball of radius $r$ centered at $x$ in $\R^n$. When the center is $0$, we often suppress the dependency on the center and simply write $B_r$. $\mathbb{S}^{n-1}$ denotes the $(n-1)$-dimensional unit sphere embedded in $\R^n$. At the risk of abusing notation, the Euclidean norm of vectors, the Hilbert-Schmidt norm of matrices, and the $n$-dimensional volume of subsets of $\R^n$ will all be denoted by $|\cdot |$. We use $\langle\cdot,\cdot\rangle$ to denote the Hilbert-Schmidt inner product between matrices. We will use $(u)_r$ to denote the average $\frac{1}{|B_r|}\int_{B_r} u$ of a given function $u\in L^1(B_r)$. Constants will be denoted by $C_0, C_1, C_2,\dots$ with dependencies given when first introduced, and subsequently suppressed. Also, note that we are implicitly assuming that $n\geq 2$. For two sets $\Omega, \ \Omega' \subset \R^n$, we write $\Omega' \Subset \Omega$ if $\Omega'$ is compactly contained in $\Omega$. We use $\mathbf{1}_E$ to denote the indicator function on a set $E$. We use the Einstein summation notation
	\begin{definition}[BMO]\label{bmo}
		We say that $f:B_1 \to \R^{n\times n}$ has bounded mean oscillation ($f\in \BMO(B_1)$) with modulus $\omega > 0$ in $B_1$ if for any ball $B\subset B_1$ we have
		\begin{align}
			\frac{1}{|B|} \int_{B}\left|f-(f)_B\right|\leq \omega\label{BMO1}
		\end{align}
		where $(f)_B$ denotes the average of $f$ over the ball $B$.
	\end{definition}

    \begin{remark}
        Note that via the John-Nirenberg inequality, condition \eqref{BMO1} implies that for any $p \in (1,\infty)$,
		\begin{align}
			\frac{1}{|B|} \int_{B}\left|f-(f)_B\right|^p\leq \bar{C}\omega\label{BMO2}
		\end{align}
    for all balls $B\subset B_1$, where $\bar{C}= \bar{C}(n,p) > 0$.
    \end{remark}
	
	Our main result is the following $\varepsilon$-regularity theorem:
	
	\begin{theorem}\label{main1}
		 Suppose that $u\in W^{2,\infty}(B_{1})$ is a critical point of
		\eqref{Ffunc} where
		$F$ is smooth and uniformly convex or uniformly concave on $U$ and $D^2u(x) \in U$ for almost every $x\in B_1$. There exists $\omega(\Lambda,n,\|D^2u\|_{L^{\infty}(B_1)})>0$ such that if $D^2u\in\BMO(B_1)$ with modulus $\omega$, then $u$ is smooth in $B_{1/2}$.
	\end{theorem}
{Note that smooth refers to $C^{\infty}$ throughout this paper.} \\

    As a consequence, we obtain the following dimension estimate on the singular set for critical points of \eqref{Ffunc}.
    \begin{corollary}\label{c:dim-est}
        Suppose that $u\in W^{2,\infty}(B_1)$ is a critical point of
		\eqref{Ffunc} where
		$F$ is smooth and uniformly convex or concave on $U$ and $D^2u(x) \in U$ for almost every $x\in B_1$. Then there exists $p_0 > 2$ such that the following holds. Let
        \[
            \Sigma(u) \coloneqq \left\{x\in B_1 : \liminf_{r\to 0} \frac{1}{r^n}\int_{B_r(x)} \left|D^2 u - (D^2 u)_{B_r(x)}\right|^{p_0} > 0 \right\},
        \]
        where $(D^2 u)_{B_r(x)}$ denotes the average of $D^2u$ over $B_r(x)$. Then $u\in C^\infty(B_1\setminus \Sigma(u))$ and
        \begin{equation}\label{e:dim}
            \dim_{\Hcal}(\Sigma(u)) \leq n-p_0.
        \end{equation}
	\end{corollary}

{As a key intermediate step towards proving Theorem \ref{main1}, we establish an improved third order Sobolev regularity estimate, which we believe is of independent interest in itself.

\begin{proposition}\label{p:higher-int-D3-intro}
		Suppose that $u\in W^{2,\infty}(B_{1})$ is a weak solution of  \eqref{main_0} on $B_{1}$ with $D^2u(x)\in U$ for almost-every $x\in B_1$, with $F$ smooth and uniformly convex or concave on $U$. Then there exists $p_0\in (2,\infty)$ such that for any compactly contained open subset $\Omega'$ of $B_1$, $u\in W^{3,p_0}(\Omega')$ and satisfies the estimate
		\begin{equation}
			\left\|u\right\|_{W^{3,p_0}(\Omega')}\leq C\left(\Lambda,\left\|b^{ij,kl}\right\|_{L^\infty(U)}, d(\Omega',\partial B_1)\right) \left\|u\right\|_{W^{2,2}(B_1)},
		\end{equation}
        where $d(\Omega',\partial B_1))$ denotes the distance $\inf_{x\in\Omega'} \mathrm{dist}(x,\partial B_1)$.
		\end{proposition}
}

The gradient of any critical point of \eqref{Ffunc} can be seen to solve a second order system, with the additional constraint that one seeks only solutions among conservative vector fields for such a system. {More precisely, if we let $v=Du$ in \eqref{Ffunc} and search for critical points of the gradient-dependent functional 
\begin{equation}\label{e:gradient-energy}
    \int_\Omega F(Dv)\, dx,
\end{equation}
among the class of gradient vector fields $w=(\partial_1 \eta, \dots,\partial_n \eta)$, the weak fourth order equation \eqref{main_0} becomes a second order system of the form
\begin{equation}\label{e:2nd-order-sys}
    \int_\Omega F^{ij}(Dv) \partial_j w_i \, dx,
\end{equation}
where now we are writing $F^{ij} = \frac{\partial F}{\partial (\partial_j v_i)}(Dv)$. Note that here $v_i$ and $w_i$ denote the components of the vector fields $v$ and $w$ respectively, and we are instead using the classical notation $\partial_j$ for derivatives. However, since we are only allowing to test against gradient vector fields in \eqref{e:2nd-order-sys}, we no do not necessarily know that we have a critical point for this second order system with an arbitrary test vector field $w$.} One may nevertheless use this observation to compare the conclusion of Corollary \ref{c:dim-est} to \cite{KristensenMingione}*{Theorem 1.1, Theorem 8.1}, where higher integrability is also exploited to get an improved dimension estimate on the singular set of local minimizers {of gradient-dependent energies such as \eqref{e:gradient-energy} (and generalizations that include additional dependencies)}. Note that the dimension bound in \cite{KristensenMingione}*{Theorem 1.1} is weaker than that demonstrated here, since therein, one is not able to establish integrability for a full additional derivative of the minimizer $u$, but rather only a fractional derivative (cf. Theorem 4.2 therein). On the other hand, in \cite{KristensenMingione}*{Theorem 8.1}, an analogous dimension estimate to \eqref{e:dim} is established for the complement of the set where a local minimizer $u$ {(with a priori $W^{1,p}$-regularity)} is merely $C^{0,\alpha}$.  

Notice that in \cite{KristensenMingione}, the authors focus on local minimizers. In fact, due to the uniform convexity assumption of our energy functional, any critical point of \eqref{Ffunc} is a minimizer relative to $W^{2,2}$ competitors with the same boundary data (cf. Remark \ref{r:bvp}).
%We investigate the impact of the uniform convexity assumption on improved regularity of critical points in forthcoming work.
Let us point out some additional important observations regarding our main results.
 
\begin{comment}
 \textcolor{blue}{In general, I don't think there is a reason to expect the results that hold good for second order systems to be exactly true for a fourth order scalar equation.}
\note{I agree with this, what I meant was more that the improved integrability holds for second order systems also, which should in turn give them this slightly sharper dimension estimate on the singular set there also (if the argument we have here is ok), despite the fact that it isn't written in Giaquinta-Martinazzi}

\end{comment}

\begin{remark}
	It might be possible to drop the a priori regularity assumption on the critical point $u$ of \eqref{Ffunc} below $W^{2,\infty}(B_1)$ (e.g. to $W^{2,2}$), but in general, such a relaxation seems unlikely without additional assumptions on $u$. See Section \ref{ss:W2infty} for a more detailed discussion.
\end{remark}

\begin{remark}
    The exponent $p_0$ in Corollary \ref{c:dim-est} {and Proposition \ref{p:higher-int-D3-intro}} comes from the application of Proposition \ref{Gehring_1}; it is the fourth order analog of the higher integrability exponent for solutions of second order divergence form elliptic PDEs.
\end{remark}

\begin{remark}\label{r:no-sing}
    We do not expect the dimension estimate of Corollary \ref{c:dim-est} to be sharp. Indeed, critical points of \eqref{Ffunc} can be seen to solve the second order system {\eqref{e:2nd-order-sys} in their gradient subject to the constraint that the test vector field $w$ is a gradient}, and solutions of second order elliptic systems may have singularities (cf. \cite{GM}*{Section 9.1}). {However, none of the existing examples therein can be realized as solutions of a fourth order equation. One can therefore ask whether the additional structure coming from the fourth order equation might give rise to improved regularity.} We investigate this problem in the forthcoming work \cite{BS4th}.
\end{remark}

In \cite{BW1}, the first author together with Warren showed that if $F$ in \eqref{Ffunc} is a smooth convex function of the Hessian and can be expressed as a function of the square of the Hessian, then a $C^{2,\alpha}$ critical point (under compactly supported variations) of \eqref{Ffunc} will be smooth. This was achieved by establishing regularity for a class of fourth order equations in the double divergence form, given by \eqref{eq1}.
In \cite{BCW}, Bhattacharya-Chen-Warren studied regularity for a certain class of fourth order equations in double divergence form that in turn led to proving smoothness for any $C^{1}$-regular Hamiltonian stationary
Lagrangian submanifold in a symplectic manifold. More recently, in \cite{ABfourth}, the first author considered variational integrals of the form \eqref{Ffunc} where $F$ is convex and smooth on the Hessian space and showed that a critical point $u\in W^{2,\infty}$ of such a functional under compactly supported variations is smooth if the Hessian of $u$ has a small $L^\infty$-oscillation. Theorem \ref{main1} relaxes the a priori assumption of uniform smallness of oscillation of the Hessian to smallness in an $L^{p_0}$-mean sense, which in turn allows one to deduce the dimension estimate of Corollary \ref{c:dim-est}. The BMO-smallness assumption on the Hessian can be compared with the smallness of the tilt-excess in Allard's regularity theorem \cite{Allard}, which is a sufficient (and generally necessary) condition to establish local regularity of minimal surfaces. \\

In this paper, the proof of the main result, Theorem \ref{main1}, follows a similar strategy as the proof of the Evans-Krylov  regularity estimate \cite{Evans82,Krylov,Krylov2}, where the proof involves first establishing $C^{1,1}$ estimates, followed by
$C^{2,\alpha}$ regularity:
$W^{2,\infty}$ leads to $C^{2,\alpha}$ estimates, from which smoothness follows from \cite{BW1}. %\note{Anna: silly question: once we have $C^{2,\alpha}$, does smoothness follow easily because of Schauder estimates? Or is there something more involved happening here?}\textcolor{blue}{- This needs a reference for sure. I have added it now.}
In \cite{ABfourth}, the assumption on the smallness of the oscillation of the Hessian was key in proving smoothness, since it lead to controlling the oscillation of the leading coefficients of the fourth order equation by the small oscillation modulus of the Hessian of $u$. This small modulus was used to bound the $L^2$ norm of the Hessian of the difference quotient of $u$ by a factor of $r^{n-2+2\alpha}$ on a ball of radius $r$. This in turn led to a uniform $C^{1,\alpha}$ bound on the difference quotient of u, thereby proving $u\in C^{2,\alpha}$ which is sufficient to achieve smoothness (see \cite{BW1}). Indeed, H\"older continuity of the Hessian leads to H\"older continuous coefficients, which leads to a self-improving solution. Then, using the convexity property of $F$, the regularity theory developed in \cite{BW1} %\note{Anna: sorry, again concerning smoothness, maybe it's a good idea to give at least one reference to the regularity theory used to get this, here?} 
is applied to the critical point $u$ to achieve its smoothness.\\

Here, however, since our a priori assumption on the smallness of the oscillation of the Hessian of $u$ is merely in a BMO sense, we are no longer able to control the $L^\infty$-oscillation of the leading coefficients of our fourth order equation by the small oscillation modulus of the Hessian of $u$. As a consequence, after exploiting an integrability improvement (see Proposition \ref{Gehring_1}), one of the main technical challenges that we face lies in controlling the ratio of the $L^p$ norms of the Hessian on balls of varying sizes by the ratio of the radii of those balls. We get around this with a key technical tool in Lemma \ref{HanLin}, which is an adaptation of \cite{HanLin}*{Lemma 3.4} and allows one to suitably absorb error terms when getting a scaled estimate for the Hessian. With the help of this, we indeed achieve a delicate estimate for our ratio bound, which eventually leads to controlling the behavior of the leading coefficients. 

The dimension estimate of Corollary \ref{c:dim-est} on the interior singular set follows as a simple consequence of Theorem \ref{main1}, a standard covering procedure, and {the higher Sobolev regularity estimate of Proposition \ref{p:higher-int-D3-intro} for $u$}, since $\Sigma(u)$ is defined precisely in terms of the failure of the smallness of the BMO-modulus for the Hessian of $u$.\\

\begin{remark}\label{r:bvp}
    Using standard techniques from the calculus of variations (namely, the Direct Method), one can establish the existence of a critical point of \eqref{Ffunc} in the space
    \[
    \mathcal{A}[g]=\{ u\in W^{2,2}(B_1) : u=g \text{ on } \partial B_1 \text{ and } Du=Dg \text{ on } \partial B_1 \} 
    \] for a given $g \in W^{3/2,2}(\partial B_1)$. However, seeking critical points $u$ of \eqref{Ffunc} with general boundary values
    \[
        \begin{cases}
            u=g & \text{on $\partial B_1$,} \\
            Du=f & \text{on $\partial B_1$},
        \end{cases}
    \]
    for any given $g\in W^{3/2,2}(\partial B_1)$, $f\in W^{1/2,2}(\partial B_1)$, is a significantly more difficult problem, and in general remains open.
\end{remark}

\begin{remark}
The second order analog of \eqref{main_0} arises by considering the Euler Lagrange formulation of variational integrals
\begin{equation*}
\int_{\Omega}F(Du)\,dx,
\end{equation*}
defined on gradient spaces. It is worth noting that regularity for critical points of the above functional does not require any restrictions on the gradient due to the well-known De Giorgi-Nash-Moser theory. \\
\end{remark} 

\subsection{Applications.}
{Consider the Hamiltonian stationary equation, which is given by 

\begin{align}
	\int_{\Omega}\sqrt{\det g}g^{ij}\delta^{kl}u_{ik}\eta_{jl}dx=0
	\text{     }\forall\eta   \in C_{0}^{\infty}(\Omega) \label{hstat}
	\end{align}
	where $g=I_n+\left(  D^{2}u\right)  ^{2}$ is the induced metric from the Euclidean metric in $\mathbb C^{n}$. This fourth order equation is of the form \eqref{main_0} and arises naturally in geometry, coming from the Euler-Lagrange equation satisfied by critical points of the volume functional on the Lagrangian submanifold $\{(x, Du(x)) \, : \, x\in \Omega\}$ with respect to compactly supported variations.

As a consequence of the main result in Theorem \ref{main1}, we have the following full regularity result for solutions of the Hamiltonian stationary equation provided that the Hessian is uniformly small enough.

\begin{corollary}\label{c:HS-reg}
    Let $\eta \in (0,1)$. Suppose that $u\in W^{2,\infty}(B_1)$ is a solution of \eqref{hstat} in $B_1$ and $\|D^2u\|_{{L^\infty}(B_1)}\leq 1-\eta$. There exists $\omega(n,\eta)>0$ such that if $D^2u\in\BMO(B_1)$ with modulus $\omega$, then $u$ is smooth in $B_{1/2}$ with interior H\"{o}lder estimates of all orders.
\end{corollary}

We discuss this application to the Hamiltonian stationary equation, together with the presence of the a priori $W^{2,\infty}$ bound, in more detail in Section \ref{h_stat}.}

While we focus on our application to the fourth order Hamiltonian stationary equation, let us mention a small subset of other interesting geometric equations to which our results may be applied. Interacting
agent problems can be described in an optimal transport framework (cf. \cite[Chapter 7]{MR3409718}) via
\begin{align*}
F({x,D^{2}u})  & =\int_\Omega \tilde{g}\left(  \det(u_{ij})\right)\,dx+\int_\Omega \Tilde{G}(x)\det
(u_{ij})\,dx\\
& +\int_\Omega\underbrace{\left[\int_\Omega W(x,y)\det(u_{ij}(y))\det(u_{ij}(x))\,dy\right]}_{{=H(D^2 u)(x)}}\,dx,
\end{align*}
where $W$ is, for example, the integrand associated with the Wasserstein distance. Here, appropriate constraints are imposed on the functions $G$, $g$, $\tilde g$, $\tilde G$, and $W$, as well as on the domain $\Omega$; we refer the reader to the references above for details. {Note that the explicit $x$-dependency in the energy functional here does not affect the Euler-Lagrange equation, so provided that this choice of $F$ is uniformly convex and has enough regularity in $x$, our results may be applied.}

In addition, Abreu's
equation \cite{Abreu, MR2154300}, {Ma-Trudinger-Wang equations} \cite{Ma-Trudinger-Wang}, and prescribed affine curvature equations \cite{MR2137978}
are derived from functionals of the form
\begin{align*}
F({x, u, D^{2}u})=\int_{\Omega}\left(  G\left(  {u_{ij}},\det(u_{ij})\right)  -{g(x)}u\right)\,dx
\end{align*}
(cf. \cite{MR3437587} for progress on these equations). {In this case, the additional dependence on $u$ of the energy functional amends the Euler-Lagrange equation: it adds in a lower-order term involving the second and third derivatives of 
$u$. Nevertheless, this can be treated in a similar manner and is, in fact, simpler to address compared to the fourth-order derivative terms. }

\subsection{Structure of paper} The organization of the paper is as follows: In Section 2, we state and prove preliminary results. In Section 3, we develop regularity theory for weak solutions of (\ref{eq1}) with small BMO-Hessian by first establishing a uniform $C^{2,\alpha}$ estimate for the solution. In Section 4, we prove our main result and its corollary. Finally, in Section 5, we state an application of our main result to the Hamiltonian stationary equation and use it to show the importance of the a priori $W^{2,\infty}$ assumption on $u$ in our main results. \\

\subsection*{Acknowledgments} The authors thank Camillo De Lellis for helpful comments. AB acknowledges
the support of NSF Grant DMS-2350290, the Simons Foundation grant MPS-TSM-00002933, and the AMS-Simons Travel Grant. AS acknowledges the support of the National Science Foundation through the grant FRG-1854147.

 \section{Preliminaries}

We begin by collecting some important preliminary results for solutions of a constant coefficient double divergence equation. The first is a higher integrability analog of \cite{BW1}*{Theorem 2.1}.

 \begin{theorem}\label{t:constant_coeff_Lp}
		Let $p_0>2$ and $r> 0$. Suppose that $w\in W^{2,p_0}(B_{r})$ satisfies the uniformly elliptic constant coefficient equation \begin{equation}
			\int_{B_r} c_{0}^{ij,kl}w_{ij}\eta_{kl}\,dx  =0 \qquad \forall\eta  \in C_{0}^{\infty}(B_{r}(0)),\label{ccoef}
		\end{equation}
		where the coefficient tensor $c_0$ satisfies the Legendre ellipticity condition
  \[
    c_0^{ij,kl}\sigma_{ij}\sigma_{kl}\geq \lambda |\sigma|^2 \quad \forall \sigma\in S^{n\times n}.
  \]
  Then there exist $C_1=C_1(n,p_0,\lambda)$, $C_2=C_2(n,p_0,\lambda)$ such that for any $0<\rho\leq r$, we have
		\begin{align*}
			\int_{B_{\rho}}\left|D^{2}w\right|^{p_0}  &  \leq C_{1}\left(\frac{\rho}{r}\right)^{n}\int_{B_{r}}\left|D^{2}w\right|^{p_0}\\
			\int_{B_{\rho}}\left|D^{2}w-(D^{2}w)_{\rho}\right|^{p_0}  &  \leq C_{2}\left(\frac{\rho}{r}\right)^{n+p_0}%
			\int_{B_{r}}\left|D^{2}w-(D^{2}w)_{r}\right|^{p_0}.
		\end{align*}
	\end{theorem}
	
	\begin{proof}
		%\note{The dilation invariance only really works if we first reduce to showing this for $w$ instead of $D^2 w$ (else the RHS gains an extra $(r/\rho)^{-2p_0}$)- this can be done since for constant coeff equation, if $w$ is a solution then so are all its derivatives, so I have amended the proof accordingly (in Han Lin they do the same). I initially thought that I could just do everything with $r$ and not set it to be 1 but the scaling didn't quite work out}
		
		We prove the desired estimates for $w$ in place of $D^2 w$; the conclusion follows immediately since whenever $w$ solves \eqref{ccoef}, so do all of its higher order derivatives. By dilation, we may consider $r=1$. We restrict our consideration to the range
		$\rho\in(0,1/4]$ noting that the statement is trivial for $\rho\in\lbrack 1/4,1]$.
		%where $a\in (0,1/2)$ is some constant.
		
		First, we note that $w$ is smooth (cf. \cite[Theorem 33.10]{Driver03}). Recall
		\cite[Section 4, Lemma 2, applied to elliptic case]{DongKimARMA}, which tells us that for a uniformly elliptic, {constant coefficient, linear} $4$th order operator $L_{0}$,
		\begin{align*}
			L_{0}w  =0\text{ in }B_{R} \qquad \implies \qquad \left\Vert Dw\right\Vert _{L^{\infty}(B_{R/4})}\leq C_{3}%
			(\lambda,n,p_0,R)\left\Vert w\right\Vert _{L^{p_0}(B_{R})}.
		\end{align*}
		In particular, we have
		\begin{equation}
			\left\Vert Dw\right\Vert _{L^{\infty}(B_{1/4})}^{p_0}\leq C_{4}(\lambda
			,n,p_0)\left\Vert w\right\Vert _{L^{\infty}(B_{1})}^{p_0}. \label{dong}%
		\end{equation}
		Therefore
		\begin{align*}
			\left\Vert w\right\Vert_{L^{p_0}(B_{\rho})}^{p_0}  &  \leq C_{5}(n)\rho
			^{n}\left\Vert w\right\Vert _{L^{\infty}(B_{1/4})}^{p_0}\\
			&  =C_{5}\rho^{n}\inf_{x\in B_{1/4}}\sup_{y\in B_{1/4}}\left\vert w(x)+w(y)-w(x)\right\vert ^{p_0}\\
			&  \leq C_{5}\rho^{n}\inf_{x\in B_{1/4}}\left[  \left\vert w(x)\right\vert+\frac{1}{2}\left\Vert
			Dw\right\Vert _{L^{\infty}(B_{1/4})}\right]  ^{p_0}\\
			&  \leq C_{6}(n,p_0)\rho^{n}\left[  \inf_{x\in B_{1/4}}\left\vert w(x)\right\vert
			^{2}+\left\Vert Dw\right\Vert _{L^{\infty}(B_{1/4})}^{p_0}\right] \\
			&  \leq C_{6}\rho^{n}\left[  \frac{1}{|B_{1/4}|}\left\|w\right\|_{L^{p_0}(B_{1/4})}%
			^{p_0}+C_{4}\left\|w\right\|_{L^{p_0}(B_{1/4})}^{p_0}\right] \\
			&  \leq C_{7}(n,p_0,\lambda)\rho^{n}\left\|w\right\|_{L^{p_0}(B_{1})}^{p_0}.
		\end{align*}
		
		Similarly,%
		\begin{align}
			\int_{B_{\rho}}\left\vert w-(w)_{\rho}\right\vert ^{p_0}  &  \leq
			\int_{B_{\rho}}\left\vert w-w(0)\right\vert ^{p_0}\nonumber\\
			&  \leq\int_{0}^{\rho}\int_{\mathbb{S}^{n-1}}\tau^{p_0}\left\vert Dw\right\vert
			^{p_0}\tau^{n-1}d\tau d\phi\nonumber\\
			&  =C_{8}(n,p_0,\lambda)\rho^{n+p_0}\left\| Dw\right\| _{L^{\infty}(B_{1/4})}^{p_0}.
			\label{fred}%
		\end{align}
		The estimate \eqref{dong} then yields
		\begin{equation*}
			\int_{B_{\rho}}\left\vert w-(w)_{\rho}\right\vert ^{p_0} \leq C_8 C_4\rho^{n+p_0}\int_{B_{1}}\left\vert w\right\vert ^{p_0}.
		\end{equation*}
		Now we may apply the  above to $D^2 w$ in place of $w$, which is possible since all second order derivatives of $w$ also solve \eqref{ccoef}, as previously mentioned, to give
		\begin{equation}\label{e:av}
			\int_{B_{\rho}}\left\vert D^2 w-(D^2 w)_{\rho}\right\vert ^{p_0} \leq C_8 C_4\rho^{n+p_0}\int_{B_{1}}\left\vert D^2 w\right\vert ^{p_0}.
		\end{equation}
		Next, observe that \eqref{ccoef} is purely fourth order, so the equation still
		holds when a second order polynomial is added to the solution. In
		particular, we may choose $\bar{w}(x) \coloneqq w(x) - \langle \left(D^{2}w\right)_{1}, x\otimes x\rangle$, so that
		\[
		D^{2}\bar{w}=D^{2}w-\left(  D^{2}w\right)  _{1}%
		\]
		also satisfies the equation. \ Then
		\[
		D^{3}\bar{w}=D^{3}w
		\]
		and so by the Poincar\'{e} inequality \eqref{dong}, we have
		\begin{equation}\label{ed}
			\int_{B_{1}}\left\vert D^{2}\bar{w}\right\vert ^{p_0}= \int_{B_{1}}\left\vert D^{2}w-\left(  D^{2}w\right)  _{1}\right\vert
			^{p_0}.
		\end{equation}
		We conclude from \eqref{e:av}, \eqref{ed} and \eqref{fred} that
		\[
		\int_{B_{\rho}}\left\vert D^{2}w-(D^{2}w)_{\rho}\right\vert ^{p_0}\leq C_{8} C_4%
		\rho^{n+2}\int_{B_{1}}\left\vert D^{2}w-\left(  D^{2}w\right)
		_{1}\right\vert ^{p_0}.
		\]
		
	\end{proof}
		
	The next result is a simple consequence of Theorem \ref{t:constant_coeff_Lp} (cf. \cite{BW1}*{Corollary 2.2} for the $p=2$ analog).
	
	\begin{corollary}\label{c:higher_Lp}
		\textit{Suppose that }$w$, $p_0$, and $r$ \textit{are as in Theorem \ref{t:constant_coeff_Lp}, and let $C_1$, $C_2$ be the constants therein.
			Then there exists $C_0=C_0(n,p_0)$ such that for any~}$u\in W^{2,p_0}(B_{r}),$ and\textit{ any~} $0<\rho\leq r,$
		there hold the following two inequalities
		\begin{align}
			\int_{B_{\rho}}\left\vert D^{2}u\right\vert ^{p_0} &\leq C_0 C_{1}\left(\frac{\rho}{r}\right)^{n}%
			\int_{B_{r}}\left\vert D^{2}u\right\vert ^{p_0}  \label{twothree} \\
            &\qquad +\left(  C_0 +C_0^2 C_{1}\right)
			\int_{B_{r}}\left\vert D^{2}(u-w)\right\vert ^{p_0} , \nonumber\\
   %\textit{\centering{and}}\nonumber\\
			\int_{B_{\rho}}\left\vert D^{2}u-(D^{2}u)_{\rho}\right\vert ^{p_0} 
			&\leq C_0^2 C_{2}\left(\frac{\rho}{r}\right)^{n+p_0}\int_{B_{r}}\left\vert D^{2}u-(D^{2}u)_{r}\right\vert
			^{p_0} \label{twofive} \\
            &\qquad+\left(  2C_0^2 +2C_0^3 C_{2}\right)  \int_{B_{r}}\left\vert D^{2}(u-w)\right\vert
			^{p_0}. \nonumber
		\end{align}
	\end{corollary}
	
	\begin{proof}
		%Once again, this follows by the same reasoning as \cite{BW1}*{Corollary 2.2}, but we repeat the proof here. 
  Let $v=u-w.$ Then \eqref{twothree} follows from direct computation and Theorem \ref{t:constant_coeff_Lp}:
		\begin{align*}
			\int_{B\rho}\left|D^{2}u\right|^{p_0}  &  \leq C_0(n,p_0)\left[\int_{B_{\rho}}\left\vert D^{2}w\right\vert^{p_0}+\int_{B_{\rho}%
			}\left|D^{2}v\right|^{p_0}\right] \\
			&  \leq C_0 C_{1}\left(\frac{\rho}{r}\right)^{n}\int_{B_{r}}\left\vert D^{2}w\right\vert^{p_0}+C_0\int_{B_{r}}%
			\left|D^{2}v\right|^{p_0}\\
			&  \leq C_0^2 C_{1}\left(\frac{\rho}{r}\right)^{n}\left[  \left\|D^{2}v\right\|_{L^{p_0}(B_{r})}^{p_0}+\left\|D^{2}%
			u\right\|_{L^{p_0}(B_{r})}^{p_0}\right]  +C_0\int_{B_{r}}\left|D^{2}v\right|^{p_0}\\
			&  =C_0^2 C_{1}\left(\frac{\rho}{r}\right)^{n}\left\Vert D^{2}u\right\Vert _{L^{p_0}(B_{r})}%
			^{p_0}+\left(C_0^2+C_0 C_{1}\left(\frac{\rho}{r}\right)^{n}\right)\left\Vert D^{2}v\right\Vert _{L^{p_0}(B_{r})}^{p_0}.
		\end{align*}
		~
		
		Similarly%
		\begin{align*}
			\int_{B\rho}\left\vert D^{2}u-(D^{2}u)_{\rho}\right\vert ^{p_0}  &  \leq
			C_0\int_{B_{\rho}}\left\vert D^{2}w-(D^{2}w)_{\rho}\right\vert ^{p_0}%
			\\
            &\qquad+C_0\int_{B_{\rho}}\left\vert D^{2}v-(D^{2}v)_{\rho}\right\vert ^{p_0}\\
			&  \leq C_0\int_{B_{\rho}}\left\vert D^{2}w-(D^{2}w)_{\rho}\right\vert ^{p_0} +2C_0^2\int_{B_{\rho}}\left\vert D^{2}v\right\vert ^{p_0}\\
			&  \leq C_0 C_{2}\left(\frac{\rho}{r}\right)^{n+p_0}\int_{B_{r}}\left|D^{2}w-(D^{2}w)_{r}\right|^{p_0} \\
            &\qquad+2C_0^2\int_{B_{\rho}}\left\vert D^{2}v\right\vert ^{p_0}\\
			&  \leq C_0 C_{2}\left(\frac{\rho}{r}\right)^{n+p_0}\left\{
			\begin{array}
				[c]{c}%
				C_0\int_{B_{r}}\left\vert D^{2}u-(D^{2}u)_{r}\right\vert ^{p_0}\\
				+C_0\int_{B_{r}}\left\vert D^{2}v-(D^{2}v)_{r}\right\vert ^{p_0}%
			\end{array}
			\right\}  \\
            &\qquad +2C_0^2\int_{B_{r}}\left\vert D^{2}v\right\vert ^{p_0}\\
			&  \leq C_0^2 C_{2}\left(\frac{\rho}{r}\right)^{n+p_0}\int_{B_{r}}\left\vert D^{2}u-(D^{2}u)_{r}%
			\right\vert ^{p_0}\\
			&  \qquad +\left(  2C_0^2 +2 C_0^3 C_{2}\left(\frac{\rho}{r}\right)^{n+p_0}\right)  \int_{B_{r}}\left\vert
			D^{2}v\right\vert ^{p_0}.
		\end{align*}
		The statement follows, noting that $\rho/r\leq1.$
	\end{proof}

We conclude this section with the following technical lemma, which is a modified version of \cite[Lemma 3.4]{han} and \cite[Lemma 5.13]{GM}. This modified technical tool is crucial for the proof of Theorem \ref{main1}.

 \begin{lemma}
\label{HanLin} \cite[Lemma 3.4]{han}. Let $\phi$ be a nonnegative and nondecreasing function on $[0,R].$ Then there exists $\theta=\theta(A,\gamma,\kappa) \in (0,1)$ such that the following holds. Suppose that
\begin{equation}\label{e:quant-monoton}
\phi(\tau)\leq A\left[  \left(  \frac{\tau}{r}\right)  ^{\kappa}%
+\varepsilon\right]  \phi(r)+Br^{\beta}%
\end{equation}
for any $0<\tau\leq \theta r\leq \theta R$ with $A,B,\beta,\kappa$ nonnegative constants
and $\beta<\kappa.$ \ Then for any $\gamma\in(\beta,\kappa),$ there exists $\varepsilon_{0}=\varepsilon_{0}(A,\kappa,\beta,\gamma)$ such that if
$\varepsilon<\varepsilon_{0}$ we have the following for all $0<\tau\leq \theta r\leq \theta R$%
\[
\phi(\tau)\leq c\left[  \left(  \frac{\tau}{r}\right)  ^{\gamma}%
\phi(r)+B\tau^{\beta}\right],
\]
where $c$ is a positive constant depending on $A,\kappa,\beta,\gamma.$ \ In
particular, for any $0<\tau\leq \theta R$ we have
\[
\phi(\tau)\leq c\left[  \frac{\phi(R)}{R^{\gamma}}\tau^{\gamma}+B\tau^{\beta}\right]
.
\]
Moreover, if $\beta = 0$, one may explicitly choose $\theta = (2A)^{-\frac{2}{\kappa}}$.
\end{lemma}

{The key modification required in this lemma, compared to its variant in \cite{han} and \cite{GM}, is that the estimate \eqref{e:quant-monoton} is only assumed to hold up to the scale $\theta r$. We are indeed only able to assume this weaker initial hypothesis when we apply this lemma in the next section.}

We omit the proof of Lemma \ref{HanLin} here, and simply refer the reader to \cite{HanLin} or \cite{GM}. Notice that the choice of scale $\tau$ such that $2A\tau^\kappa = \tau^\gamma$ in the proof therein guarantees the final conclusion of the lemma.

	\section{Regularity Theory for small BMO-Hessian}
\label{sec:small_bmo}
	In this section, we develop a regularity theory for weak solutions $u$ of the following fourth order equation in the double divergence form:
\begin{equation}
\int_{\Omega}a^{ij,kl}(D^{2}u)u_{ij}\eta_{kl}\,dx=0,\text{ }\forall\eta\in
C_{0}^{\infty}(\Omega), \label{eq1}
\end{equation}
 where $\Omega \subset \R^n$ is an open set. {The regularity theory for weak solutions to the more general uniformly convex (or concave) fourth order equation \eqref{main_0} will follow from the results in this section. We will see this in Section \ref{sec:main}.}
 
 Denoting $h_{m}=he_m$, we start by introducing the following definition.
\begin{definition}[Regular equation]
	We define equation \eqref{eq1} to be regular on $U$ when the 
	following conditions are satisfied on $U$:
	\begin{enumerate}
	    \item[(i)] The coefficients $a^{ij,kl}$ depend smoothly on $D^{2}u$.
	    \item[(ii)] The linearization of \eqref{eq1} is uniformly elliptic, namely, the leading coefficient of the linearized equation given by 
	    \begin{align}
	        b^{ij,kl}&(D^2u(x)) \label{Bdef}\\
            &\quad=\int_{0}^1\frac{\partial}{\partial{u_{ij}}} \bigg[a^{pq,kl}(D^2u(x)+t [D^2u(x+h_m)-D^2u(x)])u_{pq}(x)\bigg]dt \notag
	    \end{align}
	     satisfies the standard Legendre ellipticity condition:
	    \begin{equation}
b^{ij,kl}(\xi)\sigma_{ij}\sigma_{kl}\geq\Lambda\vert \sigma\vert^2 \quad \forall \sigma \in S^{n\times n}, \ \forall \xi\in U. 
\label{Bcondition} 
	\end{equation}

	\end{enumerate}
	\end{definition}
\begin{remark}
Observe that \eqref{eq1} is indeed regular in the sense of \cite[definition 1.1]{BW1} since for a uniformly continuous Hessian, the coefficient given by \eqref{Bdef} takes the form of the $b^{ij,kl}$ coefficient shown in \cite[(1.3)]{BW1} as $h\rightarrow 0$:
\begin{equation*}
b^{ij,kl}(D^2u(x))=a^{ij,kl}(D^2u(x))+\frac{\partial a^{pq,kl}}{\partial
	u_{ij}}(D^2u(x))u_{pq}(x) .
\end{equation*}

\end{remark}

\begin{remark}\label{r:apriori-reg}
Note that for any $\Omega' \subset \Omega$, we have
\[\left\|b^{ij,kl}\right\|_{L^\infty(\Omega')} \leq C\left(\left\|D^2 u\right\|_{L^\infty(\Omega')}\right).
\]
This is the only obstruction to removing dependency on $\left\|D^2 u\right\|_{L^\infty}$ in the constants in our techniques. Indeed, if it were possible to successfully remove such a dependency, we could relax our a priori regularity assumption on $u$ to be merely $W^{2,2}$. However, the example of the Hamiltonian stationary equation and the uniform convexity hypothesis together suggest that such a relaxation would be difficult, if at all possible; see Section \ref{ss:W2infty}.
\end{remark}

Before we proceed, let us make the following important observation, which will be crucial in the arguments that follow. For $m=1,\dots,n$, $h>0$ and a given function $g$, let $g^{h_m}$ denote the difference quotient $\frac{g(x+h_m) - g(x)}{h}$ in the direction $e_m$, recalling our previously introduced notation $h_m=h e_m$. Let $\Omega' \Subset \Omega$. For $h\leq d(\Omega',\partial\Omega)$ we may take a single difference quotient in \eqref{eq1}, which gives
\[
\int_{\Omega'}[a^{ij,kl}(D^{2}u)u_{ij}]^{h_{m}}\eta_{kl}\,dx=0.
\]
We write the first difference quotient 
as
\begin{align}
&\lbrack a^{ij,kl}(D^{2}u)u_{ij}]^{h_{m}}(x)  =a^{ij,kl}(D^{2}
u(x+h_m))\frac{u_{ij}(x+h_m)-u_{ij}(x)}{h}\nonumber\\
& \qquad\qquad +\frac{1}{h}\left[  a^{ij,kl}(D^{2}u(x+h_m))-a^{ij,kl}(D^{2}%
u(x))\right]  u_{ij}(x)\nonumber\\
& \qquad =a^{ij,kl}(D^{2}u(x+h_m))u^{h_m}_{ij}(x)\nonumber\\
& \qquad\qquad +\left[  \int_{0}^{1}\frac{\partial a^{ij,kl}}{\partial u_{pq}}%
(tD^{2}u(x+h_m)+(1-t)D^{2}u(x))dt\right]u_{ij}(x)  u^{h_m}_{pq}(x).\nonumber
\end{align}
Then, letting $f=u^{h_m}$ for $h$ sufficiently small, equation \eqref{eq1} satisfied by $u$, and definition \eqref{Bdef} of the linearized coefficient $b^{ij,kl}$ yields
\begin{equation}\label{e:linearized}
    \int_{\Omega'} b^{ij,kl} f_{ij} \eta_{kl}\,dx = 0 \qquad \forall \eta\in C_0^\infty(\Omega'),
\end{equation}
where we suppress the dependency of $b^{ij,kl}$ on $D^2 u$ in order to simplify notation.

In the remainder of this section, we assume that $B_1\subset \Omega$, which we may do without loss of generality, by rescaling and translation. Let us first present the following proposition, which will prove to be a key tool throughout. It is a fourth order analog of a standard application of the Giaquinta-Modica variant of Gehring's Lemma \cite[Proposition 5.1]{GiaMod}. 
% \note{Anna: just as a sanity check, for this, it doesn't matter what the constant coeff PDE that $w$ solves is, right? As in, we don't need to choose $c_0$ to be anything specific?}
% \textcolor{blue}{From what I understand, $c_0$ plays a role in the uniform ellipticity constant. That should be no problem at all for us.}

% \note{I guess here I was just wondering if we needed to be choosing a specific constant coeff PDE for $w$ to solve, like with the coeff being the average of $b^{ij,kl}$ or something... but I think at this point we don't need to, it comes in later with the oscillation}
% \textcolor{blue}{Right. Later we would need $c_0=(b^{ij,kl})_r$}

%\textcolor{blue}{Does it make sense to write $w$ as a solution of the constant coefficient equation $(b^{ij,kl})_r$ instead of $c_0$? That would make the constant look cleaner. But again the proposition is true for $w$ satisfying any uniformly elliptic constant coefficient equation.}
%\note{This is what I was wondering before with my comment: because it is true regardless of the constant coeff equation that $w$ solves, but somehow this result is only meaningful when we end up making a choice of $c_0$ that corresponds to freezing the coeff $b^{ij,kl}$... but maybe it is still worthwhile keeping it as it is, just in case it is useful in more generality, I'm not sure}
 \begin{proposition}\label{Gehring_1}
There exists $\bar{p}\in [1,2)$ such that the following holds. Let $w$ be as in Theorem \ref{t:constant_coeff_Lp}, let $u\in W^{2,\infty}(B_{1})$ be a weak solution of the regular equation \eqref{eq1} on $B_{1}$ with $D^2u(x)\in U$ for almost-every $x\in B_1$,
and let $v=f-w$ for $f=u^{h_m}$ as described above. Then $D^2v\in L^{p_0}_{loc}(B_1)$ and there exists $C=C(n,|c_0|,\|b^{ij,kl}\|_{L^\infty(U)})$  such that for each $x_0\in B_1$ and each $0< s < \frac{1}{2}\mathrm{dist}(x_0,\partial B_1)$, we have
\begin{equation}
    \bigg(\frac{1}{|B_{{s}}|}\int_{B_{s}(x_0)}\left|D^2v\right|^{2}\,dx\bigg)^{\frac{1}{2}}\leq C\bigg(\frac{1}{|B_{2s}|}\int_{B_{2s}(x_0)}\left|D^2v\right|^{\bar{p}}\,dx\bigg)^{\frac{1}{\bar{p}}}.\label{e:Gehring-doubling}
\end{equation}
In particular, there exists $p_0\in (2,\infty)$ such that for any $r\leq 1$ and $\rho \in (0, r)$, there exists $\bar{C}=\bar{C}(n,\Lambda, p_0, |c_0|,\|b^{ij,kl}\|_{L^\infty(U)}, r-\rho)$ such that
\begin{equation}
    \bigg(\frac{1}{|B_{{\rho}}|}\int_{B_{\rho}}\left|D^2v\right|^{p_0}\,dx\bigg)^{\frac{1}{p_0}}\leq \bar{C}\bigg(\frac{1}{|B_{r}|}\int_{B_{r}}\left|D^2v\right|^2\,dx\bigg)^{\frac{1}{2}}. \label{Gehring_2}
\end{equation}

\end{proposition}
\begin{remark}\label{r:constant-rho-r}
    It will later be necessary to know the behavior of the constant $\bar C$ with $r-\rho$ in the above proposition. Observe that $\lim_{r\uparrow \rho} \bar{C} = + \infty$.
\end{remark}
\begin{proof}
By adding a suitable constant, we may without loss of generality assume that $|v|\geq 1$ a.e. inside $B_1$. First of all, we prove \eqref{e:Gehring-doubling}. Fix $x_0$, $s$ as in the statement of the proposition. In light of \eqref{e:linearized}, we have
\[
\int_{B_{2s}(x_0)}b^{ij,kl}v_{ij}\eta_{kl}\,dx=\int_{B_{2s}(x_0)} b^{ij,kl}
  w_{ij}\eta_{kl}\,dx.
\] 
Let $\tau\in C_{c}^{\infty}\left(  B_{2s}(x_0)\right)$ be a positive cut-off function that takes the value $1$ on $B_{s}(x_0)$ and vanishes outside $B_{2s}(x)$. Note that $|D^k\tau| \leq C(n)s^{-k}$ for each $k\in \mathbb{N}$. We choose $\eta=\tau^4v$.
Expanding derivatives we get
\begin{align*}
\int_{B_{2s}(x_0)} b^{ij,kl}v_{ij}\tau^{4}v_{kl}\,dx &= \int_{B_{2s}(x_0)} b^{ij,kl}w_{ij} (\tau^4v)_{kl}\,dx \\
&\qquad- \int_{B_{2s}(x_0)}
 b^{ij,kl}v_{ij}[   (\tau^{4})
_{kl}v + (\tau^{4})_{l}v_{k}+ (\tau^{4})  _{k}%
v_{l}]\,dx.
\end{align*} 
By our assumption in \eqref{Bcondition} $ b^{ij,kl}$ is uniformly elliptic on $U$. Therefore, we get
\begin{align*}
    &\int_{B_{2s}(x_0)}\tau^{4}\Lambda\left|D^{2}v\right|^{2}\,dx \\
    &\qquad\leq C(n)s^{-2}\int_{B_{2s}(x_0)}\left\vert
 b^{ij,kl}\right\vert \left\vert v_{ij}\right\vert \tau^{2}
\left(  \left|v\right|+\left|Dv\right|\right)\,dx\\
&\qquad\qquad +Cs^{-2}\int_{B_{2s}(x_0)}\left\vert
 b^{ij,kl}\right\vert \left\vert w_{ij}\right\vert \tau^{2}
\left(  \left|v\right|+\left|Dv\right|+\left|D^2v\right|\right)\,dx\\
&\qquad\leq C\left(n,\left\|b
^{ij,kl}\right\|_{L^\infty(U)}\right)s^{-2}\int_{B_{2s}(x_0)}\left(  \varepsilon\tau^{4}\left|D^{2}v\right|^{2}+\frac
{1}{\varepsilon}\left(|v|+\left|Dv\right|\right)^{2}\right)\,dx\\
&\qquad\qquad+C\left(n,\left\|b^{ij,kl}\right\|_{L^\infty(U)},\left\|D^2w\right\|_{L^\infty(B_1)}\right)s^{-2}\int_{B_{2s}(x_0)}\left(  \varepsilon\tau^{2}\left|D^{2}v\right|^{2}+\frac
{1}{\varepsilon}\tau^{2}\right)\,dx\\
&\qquad\qquad+C\left(n,\left\|b^{ij,kl}\right\|_{L^\infty(U)},\left\|D^2w\right\|_{L^\infty(B_1)}\right)s^{-2}\int_{B_{2s}(x_0)}\left(|v|^2+\left|Dv\right|^2\right)\,dx.
\end{align*}

Note that since $w$ satisfies the constant coefficient equation \eqref{ccoef}, $\|D^2w\|_{L^{\infty}(B_1)}$ is bounded by $C\left(n,|c_0|\right)$ (see e.g. \cite[Theorem 33.10]{Driver03}).

Now we may choose $\varepsilon=\varepsilon(n,s,|c_0|,\|b^{ij,kl}\|_{L^\infty(U)},\Lambda)>0$ sufficiently small such that
\begin{align*}
    \int_{B_{2s}(x_0)}\tau^{4}\Lambda\left|D^{2}v\right|^{2}\,dx &\leq \frac{1}{2}\int_{B_{2s}(x_0)}\tau^2 \Lambda \left|D^{2}v\right|^{2}\,dx \\
    &\qquad+Cs^{-2}\int_{B_{2s}(x_0)} \left(|v|+\left|Dv\right|\right)^{2}\,dx \\
    &\qquad+Cs^{-2}\int_{B_{2s}(x_0)}\frac
{1}{\varepsilon}\tau^{2}\,dx.
\end{align*}

Rearranging, and recalling that $0\leq \tau\leq 1\leq |v|$ a.e. in $B_1$ and $\tau^4 \geq \mathbf{1}_{B_s(x_0)}$, we get
\begin{align*}
\int_{B_s(x_0)}\left|D^{2}v\right|^{2}\,dx &\leq Cs^{-2}\int_{B_{2s}(x_0)}\left(|v|^2+\left|Dv\right|^2\right)\,dx\\
&\leq C s^{-2}\bigg[\left\Vert Dv\right\Vert_{L^{2_*}(B_{2s}(x_0))}^2+\left\|D^2v\right\|_{L^{2_*}(B_{2s}(x_0))}^2\bigg]\\
&\leq Cs^{-2}\bigg(\int_{B_{2s}(x_0)}\left|D^2v\right|^{2_*}\,dx\bigg)^{\frac{2}{2_*}},
\end{align*}
and thus
\[
\frac{1}{|B_s|}\int_{B_s(x_0)}|D^{2}v|^{2}\,dx \leq C\bigg(\frac{1}{|B_s|}\int_{B_{2s}(x_0)}\left|D^2v\right|^{2_*}\,dx\bigg)^{\frac{2}{2_*}},
\]
where $C=C\left(n,|c_0|,\|b^{ij,kl}\|_{L^{\infty}(U)},\Lambda\right)$ and $2_*=\frac{2n}{n+2}$ is the exponent such that $(2_*)^* = 2$, where $p^*$ denotes the Sobolev dual exponent of $p$. The last two inequalities follow from the Gagliardo-Nirenberg inequality for bounded domains \cite{nir, gagli }. This proves the desired estimate in \eqref{e:Gehring-doubling}, with $\bar{p}=2_*$. 

It remains to prove \eqref{Gehring_2}. This, however, follows from the Giaquinta-Modica variant of Gehring's Lemma \cite{GiaMod}*{Proposition 5.1}, applied to 
$|D^2v|^{2_*}$ with $q=2$. Note that one may replace the cubes $Q_{2R}$ and $Q_R$ therein with the balls $B_r$ and $B_{\rho}$ respectively (at the price of making the constant  additionally dependent on $r-\rho$) via an analogous argument with a Calder\'{o}n-Zygmund decomposition of $B_r$ into cubes degenerating towards $\partial B_r$. We omit the details here since the argument is standard.
\end{proof}

 \begin{remark}\label{rem}
     Observe that for the above proposition to hold true, we did not require $v, Dv$ to identically vanish on the boundary of $B_r$ (since this was achieved via the choice of cut-off). In particular, notice that the conclusion of Proposition \ref{Gehring_1} holds also with $f=u^{h_m}$ in place of $v$.
 \end{remark}
 
{The arguments leading to Proposition \ref{Gehring_1} combined with Remark \ref{rem} allows us to establish the higher Sobolev regularity estimate of Proposition \ref{p:higher-int-D3-intro}. Since we first demonstrate its validity for solutions of \eqref{eq1}, we re-state it here for clarity. We will elaborate on how to conclude the version for solutions to the more general equation \eqref{main_0} in Section \ref{sec:main}.}

	\begin{proposition}\label{p:higher-int-D3}
		\label{prop_2} Suppose that $u\in W^{2,\infty}(B_{1})$ is a weak solution of the regular equation \eqref{eq1} on $B_{1}$ with $D^2u(x)\in U$ for almost-every $x\in B_1$. Then for any compactly contained open subset $\Omega'$ of $B_1$, $u\in W^{3,p_0}(\Omega')$, where $p_0$ is as in Proposition \ref{Gehring_1}, and satisfies the following estimate
		\begin{equation}
			\left\|u\right\|_{W^{3,p_0}(\Omega')}\leq C\left(\Lambda,\left\|b^{ij,kl}\right\|_{L^\infty(U)}, d(\Omega',\partial B_1)\right) \left\|u\right\|_{W^{2,2}(B_1)}. \label{est2}
		\end{equation}
%        where $d(\Omega,\partial B_1))$ denotes the distance $\inf_{x\in\Omega} \mathrm{dist}(x,\partial B_1)$.
		\end{proposition} 
		
  \begin{proof}

Let $\tau\in C_{c}^{\infty}\left(  B_{1}\right)$ be a cutoff function in
$B_{1}$ that takes the value $1$ on $\Omega'$. Fix $m\in\{1,\dots,n\}$ arbitrarily and let $\eta = \tau^4 f$ in \eqref{e:linearized}, for $f=u^{h_m}$.

This yields
\begin{equation}
\int_{B_{1}}{b}^{ij,kl}f_{ij}[\tau^{4}f]_{kl}\,dx=0, \label{dq3}%
\end{equation}
where we are suppressing the dependency of $b^{ij,kl}$ on $D^2 u$, to simplify notation. Expanding derivatives, we get
\begin{align}
\int_{B_{1}} b^{ij,kl}f_{ij}\tau^{4}f_{kl}\,dx=-\int_{B_{1}}
 b^{ij,kl}f_{ij}\left(    (\tau^{4})
_{kl}f+  (\tau^{4})  _{l}f_{k}+ (\tau^{4})  _{k}%
f_{l}\right)\,dx.\nonumber
\end{align} 
Arguing as in the proof of Proposition \ref{Gehring_1}, we get
\begin{align*}
    \int_{B_{1}}\tau^{4}\Lambda\left|D^{2}f\right|^{2}\,dx\leq C\left(\tau,D\tau,D^{2}\tau\right) \int_{B_{1}}\left\vert
 b^{ij,kl}\right\vert \left\vert f_{ij}\right\vert \tau^{2}\left(  1+|f|+|Df|\right)\,dx\\
\leq C\left(n,d(\Omega',\partial B_1)\right)\left\|b^{ij,kl}\right\|_{L^\infty(U)} \int_{B_{1}}\left(  \varepsilon\tau^{4}\left|D^{2}f\right|^{2}+C\frac
{1}{\varepsilon}\left(1+|f|+|Df|\right)^{2}\right)\,dx.
\end{align*}
Choosing $\varepsilon>0$ appropriately, rearranging, using the definition of $\tau$, and the uniform ellipticity of the coefficient $ b^{ij,kl}$ on $U$, we get
\[
\int_{\Omega'}\left|D^{2}f\right|^{2}\,dx\leq C'\int_{B_{1}}\left(1+|f|+|Df|\right)^{2}\,dx,
\]
where $C'=C'(n,\Lambda,\|b^{ij,kl}\|_{L^\infty(U)}, d(\Omega',\partial B_1))$. Now this estimate is uniform over all $h_m \leq d(\Omega',\partial B_1)$ and all directions
$e_{m}$, $m=1,\dots,n$, and so we conclude that $u\in W^{3,2}(\Omega')$, with the estimate
   
  \begin{equation*}
    \left\|u\right\|_{W^{3,2}(\Omega')}\leq C\left(n,\Lambda,\left\|b^{ij,kl}\right\|_{L^\infty(U)}, d(\Omega',\partial B_1)\right)\left\Vert u \right\Vert_{W^{2,2}(B_1)}. 
\end{equation*}
In view of Remark \ref{rem} and Remark \ref{r:apriori-reg}, the proof of the desired estimate follows verbatim from the proof of Proposition \ref{Gehring_1} by replacing $v$ with $f=u^{h_m}$.
		\end{proof}

We are now in a position to prove the main result of this section.

 \begin{theorem} \label{hold}
\label{prop_1} Suppose that $u\in W^{2,\infty}(B_{1})$ is a weak solution of the regular equation \eqref{eq1} on $B_{1}$ with $D^2u(x) \in U$ for almost-every $x\in B_1$. Let $\alpha \in (0,1)$. There exists $\omega(\Lambda,n, \alpha$, $\|D^2u\|_{L^{\infty}(B_1)})>0$ 
such that if $D^2u\in\BMO(B_1)$ with modulus $\omega$, then $D^2u\in C^{0,\alpha}(B_{3/4})$ and satisfies the following estimate
\begin{equation}
    \|D^{2}u\|_{C^{0,\alpha}(B_{3/4})}\leq C\left(\alpha,n,\Lambda,\left\|D^2 u\right\|_{L^{\infty}(B_{1})}\right). \label{est1}
\end{equation}
\end{theorem} 
	
\begin{proof}

Let $m\in \{1,\dots,n\}$ be arbitrary and let $f=u^{h_m}$ be defined in $B_{3/4}$, as before, for $h\leq \frac{1}{4}$. Recall equation \eqref{e:linearized} satisfied by $f$. We pick an arbitrary point $x_0$ inside $B_{3/4}$ and consider the ball $B_r(x_0)$, for a fixed scale $r\leq \frac{1}{4}$. To simplify notation, we will simply denote this ball by $B_r$ for the rest of this proof. Let $w$ solve the following boundary value
problem:
\begin{align}
\int_{B_{r}}(b^{ij,kl})_{r}w_{ij}\eta_{kl}\,dx  &  =0,\forall\eta\in C_{0}^{\infty
}(B_{r})\label{cons}\\
w  &  =f\text{ on }\partial B_{r}\nonumber\\
D w  &  =D f\text{ on }\partial B_{r}\nonumber.
\end{align}
This is a constant coefficient PDE with the given boundary
conditions and therefore has a unique solution $w\in W^{2,2}(B_r)$ that is smooth on the interior of
$B_{r}$ (\cite[Theorem 6.33]{Folland}).
	Letting $v=f-w$, observe that $v\in W^{2,2}_0(B_r)$ and so may be used as a test function in \eqref{cons} and \eqref{e:linearized}. Thus, we see that
	\begin{equation}
		\int_{B_r} (b^{ij,kl})_r v_{ij} v_{kl} = \int_{B_r} [(b^{ij,kl})_r - b^{ij,kl}(x)]f_{ij} v_{kl}.
	\end{equation}

	\bigskip
	By Cauchy-Schwartz, followed by H\"{o}lder's inequality, we have
	\[
	\int_{B_r} (b^{ij,kl})_r v_{ij} v_{kl} \leq \left\|(b^{ij,kl})_r - b^{ij,kl}\right\|_{L^2(B_r)}\left\|D^2 v\right\|_{L^{2p'}(B_r)}\left\|D^2 f\right\|_{L^{2p}(B_r)}.
	\]
	The ellipticity of $b^{ij,kl}$ thus gives
	\begin{equation}\label{e:Hoelder}
		\Lambda \int_{B_r} \left|D^2 v\right|^2 \leq \left\|(b^{ij,kl})_r - b^{ij,kl}\right\|_{L^2(B_r)}\|D^2 v\|_{L^{2p'}(B_r)}\left\|D^2 f\right\|_{L^{2p}(B_r)},
	\end{equation}
    where $p'$ is the H\"{o}lder dual of $p$. By Proposition \ref{Gehring_1}, for some $p_0 > 2$ and the constant $\bar{C}$ therein, we have
	\begin{equation}
		\left[\frac{1}{|B_\rho|}\int_{B_\rho}|D^2 v|^{p_0}\right]^{\frac{1}{p_0}} \leq \bar{C} \left[\frac{1}{|B_r|}\int_{B_r}|D^2 v|^{2}\right]^{\frac{1}{2}}, \label{Gehring}
	\end{equation}
	for any $\rho \in (0,r)$. Taking $p=\frac{p_0}{2} > 1$ in \eqref{e:Hoelder} gives
	\begin{align*}
		|B_r|^{1-\frac{2}{p_0}} &\left[\int_{B_\rho}\left|D^2 v\right|^{p_0}\right]^{\frac{2}{p_0}} \\
        &\leq |B_\rho|^{-\frac{2}{p_0}}|B_r| \left[\int_{B_\rho}\left|D^2 v\right|^{p_0}\right]^{\frac{2}{p_0}} \\
		&\leq \bar{C}\Lambda^{-1} \left\|(b^{ij,kl})_r - b^{ij,kl}\right\|_{L^2(B_r)}\left\|D^2 v\right\|_{L^{q_0}(B_r)}\left\|D^2 f\right\|_{L^{p_0}(B_r)},
	\end{align*}
	where $q_0=2\left(\frac{p_0}{2}\right)' = \frac{2p_0}{p_0-2}$.
	Now we have
\begin{align*}
		&\left\|(b^{ij,kl})_r - b^{ij,kl}\right\|_{L^2(B_r)} = \frac{1}{|B_r|}\left[\int_{B_r} \left|\int_{B_r} b^{ij,kl}(D^2 u(y)) - b^{ij,kl}(D^2 u(x))\,dy\right|^2\,dx \right]^{\frac{1}{2}} \\
		&\leq \frac{\left\|b^{ij,kl}\right\|_{\mathrm{Lip}(U)}}{|B_r|}\left[\int_{B_r} \left|\int_{B_r} \left|D^2 u(y)  - (D^2 u)_r\right| + \left|(D^2 u)_r - D^2 u(x)\right|\,dy\right|^2\,dx \right]^{\frac{1}{2}} \\
		&= \frac{\left\|b^{ij,kl}\right\|_{\mathrm{Lip}(U)}}{|B_r|}\left[\int_{B_r} \left|\int_{B_r} \left|D^2 u(y)  - (D^2 u)_r\right|dy + |B_r|\left|(D^2 u)_r - D^2 u(x)\right|\right|^2\,dx \right]^{\frac{1}{2}} \\
		&\leq C\left\|b^{ij,kl}\right\|_{\mathrm{Lip}(U)}\left[ \int_{B_r} \left|D^2 u(y)  - (D^2 u)_r\right|^2\,dy + \int_{B_r}\left|(D^2 u)_r - D^2 u(x)\right|^2\,dx \right]^{\frac{1}{2}} \\
		&\leq  C\left\|b^{ij,kl}\right\|_{\mathrm{Lip}(U)}\left\|(D^2 u)_r - D^2 u\right\|_{L^2(B_r)},
	\end{align*}
	for some $C=C(n)>0$. Combining the above estimates, we arrive at
	
	\begin{align*}
        &\left\|D^2 v\right\|_{L^{p_0}(B_\rho)}^{2p_0}\\
        &\qquad\leq  C\bar{C}\Lambda^{-1}r^{-np_0+2n}\left\|(D^2 u)_r - D^2 u\right\|_{L^2(B_r)}^{p_0}\left\|D^2 v\right\|_{L^{q_0}(B_r)}^{p_0}\left\|D^2 f\right\|_{L^{p_0}(B_r)}^{p_0}.\\
%        &\qquad\leq  C\bar C \Lambda^{-1} r^{-np_0+2n}\left\|(D^2 u)_r - D^2 u\right\|_{L^2(B_r)}^{p_0}\left\|D^2 v\right\|_{L^{q_0}(B_r)}^{p_0}\left\|D^2 f\right\|_{L^{p_0}(B_r)}^{p_0}
	\end{align*}
Rearranging, we arrive at
\begin{align*}
&\left\|D^2 v\right\|_{L^{p_0}(B_\rho)}^{p_0}\\
        &\qquad\leq C\bar C \Lambda^{-1} r^{-np_0+2n}\left\|(D^2 u)_r - D^2 u\right\|_{L^2(B_r)}^{p_0}\frac{\left\|D^2v\right\|^{p_0}_{L^{q_0}(B_r)}}{\left\|D^2v\right\|^{p_0}_{L^{p_0}(B_\rho)}}\left\|D^2 f\right\|_{L^{p_0}(B_r)}^{p_0}.
\end{align*}

By Corollary \ref{c:higher_Lp} (absorbing all constants into a single constant), for any $\tau\in(0,\rho)$ we thus have
	\begin{align}
		&\left\|D^2 f\right\|^{p_0}_{L^{p_0}(B_\tau)} \leq C^\dagger\left(\frac{\tau}{\rho}\right)^{n}\left\|D^2 f\right\|^{p_0}_{L^{p_0}(B_\rho)} \nonumber \\
		&\qquad +C^\dagger r^{-np_0+2n}\left\|(D^2 u)_r - D^2 u\right\|_{L^2(B_r)}^{p_0}\frac{\|D^2 v\|^{p_0}_{L^{q_0}(B_r)}}{\left\|D^2 v\right\|^{p_0}_{L^{p_0}(B_\rho)}}\left\|D^2 f\right\|_{L^{p_0}(B_r)}^{p_0} \nonumber \\
		&\leq C^\dagger \left(\frac{\tau}{\rho}\right)^{n}\left\|D^2 f\right\|^{p_0}_{L^{p_0}(B_r)} \nonumber \\
%		&\qquad +C r^{-np_0+2n}\left\|(D^2 u)_r - D^2 u\right\|_{L^2(B_r)}^{p_0}\frac{\left\|D^2 v\right\|^{p_0}_{L^{q_0}(B_r)}}{\left\|D^2 v\right\|^{p_0}_{L^{p_0}(B_\rho)}}\left\|D^2 f\right\|_{L^{p_0}(B_r)}^{p_0} \nonumber \\
		&\leq \tilde{C}^\dagger\left(\frac{\tau}{\rho}\right)^{n}\left\|D^2 f\right\|^{p_0}_{L^{p_0}(B_r)} +\tilde{C}^\dagger\omega^{p_0}\frac{\left\|D^2 v\right\|^{p_0}_{L^{q_0}(B_r)}}{\left\|D^2 v\right\|^{p_0}_{L^{p_0}(B_\rho)}}\left\|D^2 f\right\|_{L^{p_0}(B_r)}^{p_0}, \label{m1}
 	\end{align}
where $C^\dagger$ and $\tilde{C}^\dagger$ are positive constants depending on $n,\Lambda,p_0,\|D^2 u\|_{L^{\infty}(B_1)}$ and $r-\rho$. Now, we wish to apply Lemma \ref{HanLin}, with the choices of parameters
\begin{align*}
\phi(\tau) &  =\int_{B_{\tau}}\left\vert Df\right\vert ^{p_0}dx\\
A &  = \tilde{C}^\dagger\\
%\varepsilon &  = C'(\Lambda,n,p_0)\omega^{p_0}\\
\kappa & =n\\
\beta &=0, B=0\\
\gamma &  =n-p_0+p_0\alpha \in (0,\kappa),\\
R &=\frac{1}{4},
\end{align*}
for $\alpha\in (0,1)$, provided that $\omega$ is chosen to be sufficiently small.
%where the
%notations appearing on the left hand side of the above table refer to
%constants as they are named in Lemma \ref{HanLin}. 
Indeed, this can be ensured, provided that we show the following.
\begin{claim}\label{cl:bdd-ratio}
	Let $\theta$ be as in Lemma \ref{HanLin}, for the above choice of parameters. There exists $C^*=C^*(n,p_0,\theta)$ such that for any $\rho\in(\theta r, r)$, we have
	\[
		\frac{\left\|D^2 v\right\|^{p_0}_{L^{q_0}(B_r)}}{\left\|D^2 v\right\|^{p_0}_{L^{p_0}(B_\rho)}} \leq C^*.
	\]
\end{claim}

\emph{Proof of Claim \ref{cl:bdd-ratio}.}
First of all, let us write
\[
	\left\|D^2 v\right\|^{p_0}_{L^{q_0}(B_\rho)} = \left\|D^2 v \cdot\mathbf{1}_{B_\rho}\right\|^{p_0}_{L^{q_0}(B_r)}.
\]
Now, suppose for a contradiction, that the claim is false. Then we can extract a sequence of scales $\rho_k \to r$ such that
\begin{equation}\label{e:normratio-blowup}
    \frac{\left\|D^2 v\right\|^{p_0}_{L^{q_0}(B_r)}}{\left\|D^2 v \cdot\mathbf{1}_{B_{\rho_k}}\right\|^{p_0}_{L^{q_0}(B_r)}} \to \infty \qquad \text{as $k\to\infty$}.
\end{equation}
However, clearly $\mathbf{1}_{B_{\rho_k}} \to \mathbf{1}_{B_{r}}$ strongly in $L^{q_0}$, which in turn implies that
\[
    \left\|D^2 v \cdot\mathbf{1}_{B_{\rho_k}}\right\|^{p_0}_{L^{q_0}(B_r)} \to \left\|D^2 v\right\|^{p_0}_{L^{q_0}(B_r)},
\]
thus contradicting \eqref{e:normratio-blowup}.
\qed

Now that we have proved Claim \ref{cl:bdd-ratio}, we would like to fix $\rho \in (\theta r, r)$, in order to apply Lemma \ref{HanLin}. However, in order to do this we must deal with one delicate point; the constant $A$ in our above choice of parameters depends on $r-\rho$, which in turn results in the dependency of $\theta$ on $r-\rho$. Thus, we need to check that it is indeed possible to make a choice of $\rho\in (\theta r, r)$ as in the statement of Claim \ref{cl:bdd-ratio}. In light of Remark \ref{r:constant-rho-r} and the fact that we may choose $\theta=(2A)^{-\frac{2}{\alpha}}$ as stated in Lemma \ref{HanLin}, we may deduce that $\lim_{\rho\uparrow r} \theta = 0$. Thus, it is indeed possible to choose $\rho$ sufficiently close to $r$ such that $\rho > \theta r$.

Fixing such a $\rho$ and applying Claim \ref{cl:bdd-ratio}, the coefficient of the second term on the right-hand side of \eqref{m1} may now be bounded above by
\[
	\tilde{C}=(n,\Lambda,p_0,\|D^2 u\|_{L^{\infty}(B_1)})\omega^{p_0}.
\]
Thus, choosing $\omega < (\tilde{C}\varepsilon_0)^{1/p_0}$ (dependent therefore on $n,\Lambda,p_0,\|D^2 u\|_{L^{\infty}(B_1)}$), where $\varepsilon_0$ is as in Lemma \ref{HanLin}, we conclude that
\begin{align*}
\int_{B_\tau}\left|Df\right|^{p_0} dx\leq C\tau^{n-p_0+p_0\alpha}\int_{B_{1/4}}\left|Df\right|^{p_0} dx,
\end{align*}
where $C=C(\Lambda,n,\alpha,p_0)>0$. Note that the range of $\tau$ may easily be increased from $\left(0,\frac{\theta}{4}\right)$ to $\left(0,\frac{1}{4}\right)$, up to increasing the constant $c$ in Lemma \ref{HanLin}, since $B=0$ and our choice of $\phi$ is monotone non-decreasing.

Since we chose an arbitrary point $x_0\in B_{3/4}$ and $\tau\leq \frac{1}{4}$, applying Morrey's
Lemma \cite[Lemma 3, page 8]{SimonETH} to $f$ we get
\begin{align*}
\left|u^{h_m}\right|_{C^{0,\alpha}(B_{\tau}(x_0))}\leq C\left(\alpha, n,\Lambda,p_0,  \left\|u\right\|_{W^{3,p_0}(B_{1/2})} \right),
\end{align*}
 which combined with estimate \eqref{est2} gives the desired estimate \eqref{est1}. Since $p_0$ is a fixed parameter, determined by Proposition \ref{Gehring_1}, this proves the proposition.

	\end{proof}

We conclude this section with the following immediate consequence of Theorem \ref{hold} combined with the interior estimates established in \cite{BW1}.

	\begin{corollary}\label{main2}
Suppose that $u\in W^{2,\infty}(B_{1})$ is a weak solution of the regular equation \eqref{eq1} on $B_{1}$ with $D^2u(x)\in U$ for almost-every $x\in B_1$. There exists $\omega(\Lambda,n,\|D^2u\|_{L^{\infty}(B_1)})>0$ such that if $D^2u\in \BMO(B_1)$ with modulus $\omega$, then $u$ is smooth in $B_{1/2}$.
 \end{corollary}
	\begin{proof}
	    From Proposition \ref{hold} it follows that $u\in C^{2,\alpha}(B_{3/4}).$  Then smoothness follows from \cite[Theorem 1.2]{BW1}. 
	\end{proof}

	\section{Proofs of the main results}
	\label{sec:main}
 In this section, we use the results in Section 3 to prove Theorem \ref{main1} and Corollary \ref{c:dim-est}.

\subsection{Proof of {Proposition \ref{p:higher-int-D3-intro} and} Theorem \ref{main1}}
Using the results of the previous section, in particular, Corollary \ref{main2}, the proof of Theorem \ref{main1} now follows verbatim from \cite[Section 4]{ABfourth}. At the risk of being repetitive, we present the proof below for the sake of completion.

We first wish to demonstrate that the result {in Proposition \ref{Gehring_1}, and therefore Proposition \ref{p:higher-int-D3}, Theorem \ref{hold} and Corollary \ref{main2})}, is not restricted to solutions of {double divergence form equations of the type \eqref{eq1}, and in fact also applies to solutions of more general nonlinear fourth order equations} of the form \eqref{main_0}
for coefficients $F^{ij}$ that are smoothly dependent on the Hessian, as long as uniform ellipticity of its linearization (condition \eqref{Bcondition}) is maintained. In other words, we require the existence of a constant $\Lambda > 0$ such that
\begin{equation}
\frac{\partial F^{ij}}{\partial u_{kl}}(\xi)\sigma_{ij}\sigma_{kl}\geq
\Lambda\left\vert \sigma\right\vert ^{2},\text{ $\forall$ }\sigma\text{ $\in
S^{n\times n}$, $\forall \ \xi\in U$}. \label{Bb}%
\end{equation}
Arguing analogously to that in the preceding section (cf. \eqref{e:linearized}), one can check that the above observation is true by using \eqref{main_0} with a difference quotient test function (in the direction $e_m$) in order to derive the equation
\begin{equation}
\int_{\Omega'} \beta^{ij,kl}u^{h_m}_{ij}\eta_{kl} dx=0 \qquad \forall \eta\in C_0^\infty(\Omega') \label{main3}%
\end{equation}
where $\Omega'\Subset \Omega$, $h\leq d(\Omega',\partial\Omega)$ and
\begin{align*}
    \beta^{ij,kl}(D^2u(x))=\int_{0}^{1}\frac{\partial F^{ij}}{\partial u_{kl}}(D^2u(x)+t[D^2u(x+h_m)-D^2u(x)])dt. 
\end{align*}This shows that the difference quotient of the solution of \eqref{main_0} satisfies an equation of the form \eqref{e:linearized}, which was previously derived from \eqref{eq1}. Since the equation that we work with is the one satisfied by the difference quotient, the conclusion {of Proposition \ref{Gehring_1} indeed holds for solutions of \eqref{main_0}, and therefore so does the conclusion of Proposition \ref{p:higher-int-D3} for such solutions, which immediately yields Proposition \ref{p:higher-int-D3-intro}. In turn, we deduce that} Theorem \ref{hold}, and thus also Corollary \ref{main2}, holds good for solutions of \eqref{main_0}.

We are now in a position to conclude the validity of Theorem \ref{main1} from this. Observe that any critical point of \eqref{Ffunc} solves \eqref{main_0} with $F^{ij} = \frac{\partial F(D^{2}u)}{\partial u_{ij}}$. In light of the above discussion, we may apply Corollary \ref{main2}. If $F$ is uniformly convex, clearly this choice of $F^{ij}$ satisfies condition \eqref{Bb}. If, on the other hand, $F$ is uniformly concave, we may simply replace $F$ with $-F$.\qed

We are now ready to prove the dimension estimate on the singular set in Corollary \ref{c:dim-est}.

\subsection{Proof of Corollary \ref{c:dim-est}}
Let $B_r(x) \subset B_1$. By the Poincar\'{e} inequality, we have
\[
    \frac{1}{r^n}\int_{B_r(x)} \left|D^2 u - (D^2 u)_{B_r(x)}\right|^{p_0} \leq \frac{1}{r^{n-p_0}} \int_{B_r(x)} \left|D^3 u\right|^{p_0}.
\]
Thus, for every $x\in \Sigma(u)$, we have
\[
    \liminf_{r\to 0} \frac{1}{r^{n-p_0}} \int_{B_r(x)} \left|D^3 u\right|^{p_0} > 0.
\]
Now, by Proposition \ref{p:higher-int-D3}, we know that $u\in W^{3,p_0}_{\mathrm{loc}}(B_1)$. Thus, we may apply \cite{GM}*{Proposition 9.21} to conclude the desired dimension estimate.\qed

\section{Hamiltonian Stationary Equations}\label{h_stat}

Hamiltonian stationary Lagrangian submanifolds of the complex Euclidean space are critical points of the volume functional under Hamiltonian variations, and locally they are governed by a fourth order nonlinear elliptic equation, given by
\begin{equation}\label{e:HamStat}
\Delta_g \Theta=0,
\end{equation} 
where $\Delta_g$ is the Laplace-Beltrami operator on the Lagrangian graph $L_u=(x,Du)$. The function $\Theta$ is called the Lagrangian phase or angle of the surface $L_u$ and is defined by
\[
\Theta=\sum_{i=1}^n \arctan\lambda_i
\]
where $\lambda_i$ are the eigenvalues of the Hessian $D^2u$.

Let us describe the analytic setup of the geometric variational problem. For a fixed bounded domain $\Omega\subset \mathbb R^n$, let $u:\Omega\rightarrow \mathbb{R}$ be a smooth function. The gradient graph $L_u=\{(x,Du(x)): x\in \Omega\}$ is a Lagrangian $n$-dimensional submanifold in $\mathbb C^n$, with respect to the complex structure $J$ defined by the complex co-ordinates $z_j=x_j+i y_j$ for $j=1,...,n.$ The volume functional on $L_u$ is given by 

\begin{equation}
\int_{\Omega}\sqrt{\det(I_n+\left(  D^{2}u\right)  ^{2})}\,dx. \label{HS3}
\end{equation}

A function $u\in W^{2,\infty}(\Omega)$ is a critical point of this functional under compactly supported variations if and only if $u$ satisfies the Euler-Lagrange equation \eqref{hstat}.
% \begin{align}
% 	\int_{\Omega}\sqrt{\det g}g^{ij}\delta^{kl}u_{ik}\eta_{jl}dx=0
% 	\text{     }\forall\eta   \in C_{0}^{\infty}(\Omega) \label{hstat}
% 	\end{align}
% 	where $g=I_n+\left(  D^{2}u\right)  ^{2}$ is the induced metric from the Euclidean metric in $\mathbb C^{n}$.
This is also known as the variational Hamiltonian stationary equation.  
If the Hessian of the potential $u$ is bounded by a small dimensional constant, then the variational equation \eqref{hstat} is equivalent to the geometric Hamiltonian stationary equation (\ref{e:HamStat}) \cite[Theorem 1.1]{ChenWarren} (also see \cite{Oh, SW03}).

In $\mathbb{C}^{n}$ with the standard K\"ahler structure, the above 
expression for $\Theta$ is available for a local graphical
representation $L_u$ as above. This decomposition feature of the fourth order elliptic
operator into a composition of two second order elliptic operators as in \eqref{e:HamStat} is essential in the work of
Chen-Warren \cite{CW}, in which it is shown that a $C^{1}$-regular Hamiltonian
stationary Lagrangian submanifold in $\mathbb{C}^{n}$ is real analytic.
However, it is rather difficult to apply the same strategy on a Calabi-Yau manifold other than $\mathbb{C}^{n}$. Even while $\Theta$ is still locally well-defined on 
$\Omega$, it can no longer be written in a clean form as a sum of
arctangent functions, even when representing $L_u$ as a gradient graph in a Darboux coordinate chart. In \cite{BCW}, the authors found that directly dealing with a critical point of
the volume of $L_u$ in an open ball $B\subset\mathbb{R}^{2n}$ equipped
with a Riemannian metric, among nearby competing gradient graphs $L_{u_t}=\{(x, Du(x)
+t D\eta(x)):x\in\Omega\}$ for compactly supported smooth functions $\eta$ is helpful. This leads the authors to
study a class of fourth order nonlinear equations (cf. \cite{BCW}*{(1.1)}) sharing a similar structure with \eqref{eq1}. It is worth noting that in general, equations of the form \eqref{eq1} do not necessarily factor into a composition of second order operators, thus motivating the study of such fourth order equations. 

{Let us now prove Corollary \ref{c:HS-reg}, which is an important application of our main result in Theorem \ref{main1}. }
% As an application of the main result in Theorem \ref{main1}, we state the following corollary: 
% \begin{corollary}\label{c:HS-reg}
%     Let $\eta \in (0,1)$. Suppose that $u\in W^{2,\infty}(B_1)$ is a critical point of \eqref{HS3} in $B_1$ and $\|D^2u\|_{{L^\infty}(B_1)}\leq 1-\eta$. There exists $\omega(n,\eta)>0$ such that if $D^2u\in\BMO(B_1)$ with modulus $\omega$, then $u$ is smooth in $B_{1/2}$ with interior H\"{o}lder estimates of all orders.
% \end{corollary}

\subsection{Proof of Corollary \ref{c:HS-reg}}
    We show that when $\|D^2u\|_{{L^\infty}(B_1)}\leq 1-\eta$, the area functional \eqref{HS3} is uniformly convex. Then the result follows immediately from Theorem \ref{main1}.

    Suppose that the Hessian $D^2u$ is diagonalized at $x_0$, with eigenvalues $\lambda_1,\dots,\lambda_n$. We may thus write the  area integrand at $x_0$ as
    \[
    V(D^2 u(x_0)) \equiv V(\lambda_1,\dots,\lambda_n)=\Bigg[(1+\lambda_1^2)...(1+\lambda_n^2)\bigg]^{\frac{1}{2}}.
    \]

We compute
\begin{align*}
    \partial_i V&=\frac{\lambda_i}{1+\lambda_i^2}V\\
    \partial_{ii} V&=\frac{1}{(1+\lambda_i^2)^2}V\\
    \partial_{ij}V&=\frac{\lambda_j\lambda_i}{(1+\lambda_j^2)(1+\lambda_i^2)}V \text{ for $i\neq j$}.
\end{align*}
Denoting $e_i=\frac{\lambda_i}{1+\lambda_i^2}$, we re-write
\begin{align*}
    \partial_{ii} V=\bigg(\frac{1}{1+\lambda_i^2}-2e_i^2\bigg)V +e_i^2V\\
    \partial_{ij}V=Ve_ie_j  \text{ for $i\neq j$}.
\end{align*}
Hence, %\note{I'm not sure I see how you got the lower bound here - probably I just need to stare at it a little longer}
 \begin{align*}
     \frac{V_{ij}}{V}=e_ie_j+\delta_{ij}\bigg(\frac{1}{1+\lambda_i^2}-2e_i^2\bigg)
     \geq \frac{1-\lambda_i^2}{(1+\lambda_i^2)^2}\geq C(\eta)
\end{align*}
where the last lower bound follows from using the fact that $\|D^2u\|_{L^\infty}\leq 1-\eta$.\qed

\begin{remark}
    Note that the result in \cite[Theorem 1.1, Theorem 1.2]{ChenWarren} reaches a similar conclusion by considering solutions $u\in C^{1,1}(B_1)$ of the geometric equation (\ref{e:HamStat}) but under the assumption that there exists a $c(n)>0$ for which $\|u\|_{C^{1,1}(B_1)}\leq c(n)$.
\end{remark}

\subsection{A priori $W^{2,\infty}$ assumption}\label{ss:W2infty}
     We use the Hamiltonian stationary equation and the area functional to {further elaborate on the a priori regularity assumption $u\in W^{2,\infty}(B_1)$, and attempt to convince the reader that the question of whether or not this assumption is natural is a delicate one.}
     %show that it cannot easily be weakened, without making additional assumptions on $u$. 
     
     Let $u$ be a solution of \eqref{hstat}. It is currently known that \eqref{hstat} is equivalent to \eqref{e:HamStat} \emph{if} the $C^{1,1}$ norm of $u$ is sufficiently small, see \cite{ChenWarren}. This in itself demonstrates the importance of the $W^{2,\infty}$ assumption in our results, since equation \eqref{e:HamStat} has a more favorable structure than \eqref{hstat}, in the sense that the fourth order operator in \eqref{e:HamStat} factors into two second order operators. In addition, the $W^{2,\infty}$-regularity is necessary to make sense of the additional requirement $\|D^2 u\|_{L^\infty}\leq 1-\eta$, which is the sufficient condition demonstrated here for the uniform convexity of the functional \eqref{HS3}.
     
     On the other hand, if one directly assumes that $u$ is a weak solution of the Hamiltonian stationary equation \eqref{e:HamStat} in $B_1$, then using integration by parts, for any test function $\eta\in C^{\infty}_0(B_1)$, we may write: 
     \begin{align*}
         \int_{B_1}\Theta \Delta_g\eta \,d\mu_g=0.
     \end{align*}
     Recall that the Laplace-Beltrami operator of the metric $g$ on $L_u$ is given by:
\begin{align*}
\Delta_g  &=\frac{1}{\sqrt{ g}}\partial_i(\sqrt{ g}g^{ij}\partial_j ) \\
&=g^{ij}\partial_{ij}+\frac{1}{\sqrt{g}}\partial_i(\sqrt{g}g^{ij})\partial_j \label{2!}\\
&=g^{ij}\partial_{ij}-g^{jp}\Theta_q u_{pq} \partial_j. \nonumber
\end{align*}
We re-write the distributional equation as 
\begin{align*}
         \int_{B_1}\Theta \bigg(g^{ij}\partial_{ij}\eta-g^{jp}\Theta_q u_{pq} \partial_j\eta\bigg)\,d\mu_g=0 .
\end{align*}
For the above equation to be well-defined for $u\in W^{2,p}(B_1)$, with $1\leq p < \infty$, we require $\Theta \in W^{1,p'}(B_1)$, where $p'$ is the H\"{o}lder dual of $p$, satisfying $\frac{1}{p}+\frac{1}{p'}=1$. However, $\Theta \in W^{1,p'}(B_1)$ requires
$u$ to necessarily have more a priori regularity than merely $W^{2,p}$. It may not be necessary to require that $u\in W^{3,p'}(B_1)$ a priori, but this nevertheless imposes a nonlinear anisotropic third order integrability constraint on $u$.

It remains possible that when considering a critical point of \eqref{hstat}, one may relax the assumption of $W^{2,\infty}$ to $W^{2,2}$, for example. This, however, would require a weaker sufficient condition than the hypothesis $\|D^2 u\|_{L^\infty}\leq 1-\eta$ given in Corollary \ref{c:HS-reg}, in order to deduce uniform convexity of the functional \eqref{HS3}. 

{By comparison to vectorial solutions to the classical area functional, recalling that the functional \eqref{HS3} is precisely the area functional for gradient graphs $x\mapsto (x,Du(x))$, we notice that Lipschitz regularity is assumed a priori in some settings, for instance in the context of the Lawson-Osserman conjecture (see \cite{LO,HMT}), and is exploited in the arguments therein, suggesting that for some results it is a necessary hypothesis. Furthermore, counterexamples to regularity for vectorial critical points of (quasiconvex) gradient-dependent energy functionals, are often constructed via convex integration techniques, yielding Lipschitz regularity, but failure of $C^1$ regularity at singular points (see e.g. \cite{Muller-Sverak}). Despite the fact that such counterexamples \emph{cannot} be produced for uniformly convex energies, this still loosely suggests that the $W^{2,\infty}$ assumption may be natural in our setting, since the analogues of these constructions herein should yield this regularity. On the other hand, the work \cite{Sverak-Yan} demonstrates the existence of $W^{1,2}$ minimizers of uniformly convex gradient-dependent energy functionals that are not Lipschitz. If it were possible to construct analogous examples that are gradients of a scalar function, this would therefore show that the $W^{2,\infty}$ assumption, although natural, is indeed restrictive in our context.}
%for which there exist critical points that are not smooth in sufficiently high dimensions,  a uniform bound on the Hessian appears to be a reasonable condition to impose to guarantee uniform convexity, without which one does not typically expect regularity for critical points.

	\bibliographystyle{amsalpha}

\end{document}